\documentclass[twoside,12pt,leqno]{article}
\usepackage{amsmath, amsthm, amsfonts, amssymb, color}
\usepackage{mathrsfs}
\usepackage{fancyhdr}

\newcommand{\be}{\begin{eqnarray}}
\newcommand{\ee}{\end{eqnarray}}
\newcommand{\ce}{\begin{eqnarray*}}
\newcommand{\de}{\end{eqnarray*}}

\newtheorem{thm}{Theorem}[section]

\newtheorem{lem}[thm]{Lemma}

\newtheorem{exa}[thm]{Example}

\theoremstyle{definition}

\definecolor{wco}{rgb}{0.5,0.2,0.3}

\numberwithin{equation}{section} \theoremstyle{remark}
\newtheorem{rem}{Remark}[section]

\def\<{\langle} \def\>{\rangle}  
    
\def\d{\text{\rm{d}}}   
\def\E{\mathbb E}  
\def\beg{\begin} \def\beq{\begin{equation}}  \def\F{\mathcal F}

 \def\P{\mathbb P} 
 
 \def\ee{\varepsilon}

\def\[{{\Big[}}
\def\]{{\Big]}}

\def\({{\Big(}}
\def\){{\Big)}}

\allowdisplaybreaks[1]

\title{{\bf  SPDE in Hilbert Space with Locally Monotone Coefficients} 
\footnote{Please cite as: J. Funct. Anal. 259 (2010),  2902--2922.}}
\author{{\bf  Wei Liu $^{a}$\footnote{Corresponding author:
wei.liu@uni-bielefeld.de}
~
and
 Michael R\"{o}ckner $^{a,b}$
}\\
{\footnotesize $a.$ Fakult\"at f\"ur Mathematik, Universit\"at Bielefeld,
D-33501 Bielefeld, Germany}\\
  \footnotesize{$b.$ Department of Mathematics and Statistics, Purdue University, West Lafayette, 47906 IN, USA}\\
}
\date{}
\begin{document}
\maketitle


\begin{abstract}
The aim of this paper is to extend the usual framework of SPDE with
monotone coefficients to include a large class of cases with merely
locally monotone coefficients. This new framework is conceptually
not more involved than the classical one, but includes many more
fundamental examples not included previously. Thus our main result
can be applied to various types of SPDEs such as stochastic
reaction-diffusion equations, stochastic Burgers type equation,
stochastic 2-D Navier-Stokes equation, stochastic $p$-Laplace
equation and stochastic porous media equation with  non-monotone
perturbations.
\end{abstract}

\bigbreak

\section{Introduction}

Let
$$V\subset H\equiv H^*\subset V^*$$
 be a Gelfand triple, $i.e.$  $(H, \<\cdot,\cdot\>_H)$ is a separable
Hilbert space and identified with its dual space by the Riesz
isomorphism, $V$ is a reflexive  Banach space such that it is
 continuously and densely embedded into $H$. If  ${ }_{V^*}\<\cdot,\cdot\>_V$ denotes  the dualization
between  $V$ and its dual space $V^*$,  then it follows that
$$ { }_{V^*}\<u, v\>_V=\<u, v\>_H, \  u\in H ,v\in V.$$
Let $\{W_t\}_{t\geq0}$ be a cylindrical Wiener process on a
separable
 Hilbert space $U$
w.r.t a complete filtered probability space
$(\Omega,\mathcal{F},\mathcal{F}_t,\mathbb{P})$ and $\left(L_2(U;H),
 \|\cdot\|_2\right) $  denotes
the space of all Hilbert-Schmidt operators from $U$ to $H$. We
consider the following stochastic evolution equation
 \begin{equation}\label{SEE}
\d X_t=A(t,X_t) \d t+B(t,X_t) \d W_t,
\end{equation}
 where for some fixed time $T$
$$A: [0,T]\times V\times \Omega\to V^*;\  \  B:
[0,T]\times V\times \Omega\to L_{2}(U;H)$$ are progressively
measurable, $i.e.$ for every $t\in[0,T]$, these maps restricted to
$[0,t]\times V\times \Omega$ are
$\mathcal{B}([0,t])\otimes\mathcal{B}(V)\otimes\F_t$-measurable
(where $\mathcal{B}$ denotes the corresponding Borel
$\sigma$-algebra).

It is well known that (\ref{SEE}) has a unique solution if $A,B$
satisfy the classical monotone and coercivity conditions (cf.
\cite{KR79,PR07}), which we recall in the Appendix below. The theory
of monotone operators starts from  substantial work of Minty
\cite{Mi62,Mi63} and Browder \cite{Bro63,Bro64} for PDE.  We refer
to \cite{Li72,Z90,Sh97} for a detailed exposition and references. In
recent years, this variational approach has been also used
intensively for analyzing SPDE driven by infinite-dimensional Wiener
process. Unlike the semigroup approach (cf. \cite{DPZ92}), it is not
necessary to have a linear operator in the drift part which has to
generate a semigroup. Hence the variational approach can be used to
investigate
 nonlinear SPDE which are not necessarily of semilinear type. For general
results on the existence and uniqueness of solutions to SPDE we
refer to \cite{Par75,KR79,G,RRW,Zh08}. Within this framework many
different types of properties have already been established, $e.g.$
see \cite{C92,L08b,RZ,RWW} for the small noise large deviation
principle, \cite{GM07,GM09} for discretization approximation schemes
to the solutions of SPDE, \cite{W07,LW08,L08} for the
dimension-free Harnack inequality and resulting ergodicity,
compactness and contractivity properties of  the associated
transition semigroups, and \cite{L10,BGLR,GLR} for the invariance of
subspaces and existence of random attractors for corresponding
random dynamical systems. As one typical example of SPDE in this
framework, the stochastic porous media equation has been extensively
studied in \cite{BDR07,BDR08,BDR2,BDR1,DRRW,RW}.

The main aim of this paper is to provide a more general framework
for the variational approach, being conceptually not more
complicated than the classical one (cf. \cite{KR79}), but including
a large number of new applications as e.g. fundamental SPDE as the
stochastic   2-D Navier-Stokes equation and stochastic Burgers type
equation. The main changes consist of localizing the monotonicity
condition and relaxing the growth condition. This new framework is,
in addition, more stable with respect to perturbations. We refer to
Section 3 below for details. In particular, we can simplify the
related approach to the stochastic 2-D Navier-Stokes equation in the
nice paper \cite{MS02}, which inspired us a lot to start this work.
However, our approach also easily covers the case of arbitrary
multiplicative noise, whereas in \cite{MS02} only additive noise was
considered. It is also straight forward to extend our new framework
to more noise terms, e.g. Levy noise (cf. \cite{G} for the classical
case). This and further new applications will be the subject of
future work.

Let us now state the precise conditions on the coefficients of
(\ref{SEE}):

Suppose  there exist constants
  $\alpha>1$, $\beta\ge 0$, $\theta>0$, $K$ and a positive adapted process $f\in
L^1([0,T]\times \Omega; \d
    t\times \mathbb{P})$ such that the
 following
 conditions hold for all $v,v_1,v_2\in V$ and $(t,\omega)\in [0,T]\times \Omega$.
\begin{enumerate}
    \item [$(H1)$] (Hemicontinuity) The map
     $ s\mapsto { }_{V^*}\<A(t,v_1+s
 v_2),v\>_V$
  is  continuous on $\mathbb{R}$.

    \item [$(H2)$] (Local monotonicity)
$$2 { }_{V^*}\<A(t,v_1)-A(t,v_2), v_1-v_2\>_V
    +\|B(t,v_1)-B(t,v_2)\|_{2}^2\\
     \le \left(K + \rho(v_2) \right)\|v_1-v_2\|_H^2,$$
where $\rho: V\rightarrow [0,+\infty)$ is a measurable function
and locally bounded in $V$.

\item [$(H3)$] (Coercivity)
    $$ 2 { }_{V^*}\<A(t,v), v\>_V +\|B(t,v)\|_{2}^2 +\theta
    \|v\|_V^{\alpha} \le f_t +K\|v\|_H^2.$$

\item[$(H4)$] (Growth)
$$ \|A(t,v)\|_{V^*}^{\frac{\alpha}{\alpha-1}} \le (f_t +
 K\|v\|_V^{\alpha} ) ( 1 +\|v\|_H^{\beta} ).$$
    \end{enumerate}

\begin{rem} (1) $(H2)$ is essentially weaker than the standard monotonicity  $(A2)$ (i.e. $\rho\equiv0$).
 One typical form of $(H2)$ in applications  is
$$\rho(v)=C\|v\|^\gamma,$$
where $\|\cdot\|$ is some norm on $V$ and $C,\gamma$ are some
constants.

One typical example is the stochastic 2-D  Navier-Stokes equation on
a bounded or unbounded domain, which satisfies $(H2)$ but does not
satisfy $(A2)$ (see Section 3). In fact, if $A(t,v)=\nu P_H\Delta
v-P_H \left[(v\cdot \nabla) v \right]$,
 we have
$$2 { }_{V^*}\<A(t,v_1)-A(t,v_2), v_1-v_2\>_V
     \le -\frac{\nu}{2}\|v_1-v_2\|_V^2+ \left(\nu +\frac{16}{\nu^3}\|v_2\|^4_{L^4}  \right)\|v_1-v_2\|_H^2.$$

(2) If the noise is zero or additive type in (\ref{SEE}),
then the existence and uniqueness of solutions to (\ref{SEE}) can be established  by
replacing $(H2)$ with the following more general type of local
monotonicity:
$$ { }_{V^*}\<A(t,v_1)-A(t,v_2), v_1-v_2\>_V
     \le \left(K +\eta(v_1)+ \rho(v_2) \right)\|v_1-v_2\|_H^2,$$
where $\eta,\rho: V\rightarrow [0,+\infty)$ are measurable functions
and locally bounded in $V$.

This will be investigated in a separated paper \cite{L}.

(3) $(H4)$ is also weaker than the  standard growth condition $(A4)$
(see the Appendix) assumed in the literature (cf.
\cite{KR79,Z90,PR07}). The advantage of $(H4)$ is, $e.g.$, to
include many semilinear type equations with nonlinear perturbation
terms. For example, if we consider a reaction-diffusion type
equation, $i.e.$  $A(u)=\Delta u+F(u)$, then for verifying  $(H3)$
we have $\alpha=2$. Hence $(A4)$ would imply that $F$ has at most
linear growth. However, we  can allow $F$ to have some polynomial
growth by using the weaker condition $(H4)$ here. We refer to
Section 3 for more details.
\end{rem}

\beg{defn} (Solution of SEE) A continuous $H$-valued
$(\F_t)$-adapted process $\{X_t\}_{t\in [0,T]}$ is called a
solution of $(\ref{SEE})$, if for its $\d t\otimes \P$-equivalent
class $\bar{X}$ we have
$$\bar{X}\in L^\alpha([0,T]\times \Omega, \d t\otimes\P; V)\cap L^2([0,T]\times \Omega, \d t\otimes\P; H)$$
and $\P-a.s.$,
$$X_t=X_0+\int_0^t A(s, \bar{X}_s)\d s+\int_0^t B(s, \bar{X}_s)\d W_s,\ t\in[0,T].  $$
\end{defn}

Now we can state the main result.

\beg{thm}\label{T1}
 Suppose $(H1)-(H4)$ hold for  $f\in L^{p/2}([0,T]\times \Omega; \d
    t\times \mathbb{P})$ with some $p\ge \beta+2$, and
 there exists a constant $C$  such that
\begin{equation}\begin{split}\label{c3}
& \|B(t,v)\|_2^2 \le C(f_t+\|v\|_H^2), \ t\in[0,T], v\in V; \\
 & \rho(v) \le C(1+\|v\|_V^\alpha) (1+\|v\|_H^\beta), \
v\in V.
\end{split}
\end{equation}
    Then for any $X_0\in L^{p}(\Omega\rightarrow H; \mathcal{F}_0;\mathbb{P})$
    $(\ref{SEE})$
    has a unique solution $\{X_t\}_{t\in [0,T]}$ and satisfies
$$\E\left(\sup_{t\in[0,T]}\|X_t\|_H^p+\int_0^T \|X_t\|_V^\alpha \d t \right)   < \infty.$$
Moreover, if
 $A(t,\cdot)(\omega), B(t,\cdot)(\omega)$  are independent of $t\in[0,T]$ and $\omega\in \Omega$,
then  the  solution $\{X_t\}_{t\in[0,T]}$ of (\ref{SEE}) is a Markov process.

\end{thm}

\section{Proof of the main theorem}

The first step of the proof  is mainly based on the Galerkin approximation.
Let
$$\{e_1,e_2,\cdots \}\subset V$$ be an orthonormal basis of $H$ and
 let $H_n:=span\{e_1,\cdots,e_n\}$ such that $span\{e_1,e_2,\cdots\}$ is dense in $V$. Let $P_n:V^*\rightarrow H_n$ be defined by
$$ P_ny:=\sum_{i=1}^n { }_{V^*}\<y,e_i\>_V e_i, \ y\in V^*.  $$
Obviously, $P_n|_H$ is just the orthogonal projection onto $H_n$ in $H$ and we have
$$ { }_{V^*}\<P_nA(t,u), v\>_V=\<P_nA(t,u),v\>_H={ }_{V^*}\<A(t,u),v\>_V, \ u\in V, v\in H_n.  $$
Let $\{g_1,g_2,\cdots \}$ be an orthonormal basis of $U$ and
$$ W^{(n)}_t:=\sum_{i=1}^n\<W_t,g_i\>_U g_i=\tilde{P}_n W_t, $$
where $\tilde{P}_n$ is the orthogonal projection onto $span\{g_1,\cdots,g_n\}$ in $U$.

Then for each
finite $n\in \mathbb{N}$  we consider the following stochastic
equation on $H_n$
\begin{equation}\label{approximation}
\d X_t^{(n)}=P_n A(t,X_t^{(n)}) \d t+ P_n B(t,X_t^{(n)}) \d
W_t^{(n)},\ X_0^{(n)}=P_n X_0.
\end{equation}
By the classical result for the solvability of SDE in finite-dimensional space (cf. \cite{KR79,PR07}) we know that
(\ref{approximation}) has a unique  strong solution.

In order to construct the solution of (\ref{SEE}), we need some a
priori estimates for $X^{(n)}$. For  convenience we use following
notations:
\begin{equation*}
\begin{split}
 &K=L^\alpha([0,T]\times \Omega\rightarrow V; \d t\times \P);\\
&K^*=L^{\frac{\alpha}{\alpha-1}}([0,T]\times \Omega\rightarrow V^*; \d t\times \P);\\
&J=L^2([0,T]\times \Omega\rightarrow L_2(U;H); \d t\times \P).
\end{split}
\end{equation*}

\begin{lem}\label{L1}
Under the assumptions in Theorem \ref{T1}, there exists $C>0$ such
that for all $n\in \mathbb{N}$
$$ \|X^{(n)}\|_K + \sup_{t\in[0,T]} \E \|X^{(n)}_t\|_H^2 \le C.     $$
\end{lem}
\begin{proof} The conclusion follows from
 $(H3)$ by using  the same argument as in \cite[Lemma 4.2.9]{PR07}.
Hence we omit the details here.
\end{proof}

\begin{lem}\label{L2}
 Under the assumptions in Theorem \ref{T1}, there exists $C>0$ such
that for all $n\in \mathbb{N}$
 we have
\begin{equation}\label{l2}
\E\sup_{t\in[0,T]}\|X_t^{(n)}\|_H^p +\E\int_0^T \|X_t^{(n)}\|_H^{p-2}\|X_t^{(n)}\|_V^\alpha \d t
 \le C\left(\E\|X_0\|_H^p+ \E\int_0^T f_t^{p/2}\d t\right).
\end{equation}
In particular, there exists $C>0$ such
that for all $n\in \mathbb{N}$
$$  \|A(\cdot, X^{(n)})\|_{K^*} \le C.     $$
\end{lem}

\begin{proof}

 By   It\^{o}'s formula,  Young's inequality and (\ref{c3}) we have
\beq\begin{split}\label{Ito estimate 2}
\|X_t^{(n)}\|_H^p=&\|X_0^{(n)}\|_H^p
+p(p-2)\int_0^t\|X_s^{(n)}\|_H^{p-4}\|
(P_nB(s,X_s^{(n)})\tilde{P}_n)^* X_s^{(n)}\|_{H}^2 \d s\\
&+\frac{p}{2}\int_0^t\|X_s^{(n)}\|_H^{p-2}\left( 2 {
}_{V^*}\<A(s,X_s^{(n)},X_s^{(n)}\>_V+\|P_nB(s,X_s^{(n)})\tilde{P}_n\|_{2}^2\right)
\d s
\\
&+p\int_0^t\|X_s^{(n)}\|_H^{p-2}\<X_s^{(n)},P_n B(s,X_s^{(n)})\d
W_s^{(n)}\>_H\\
\le & \|X_0\|_H^p- \frac{p\theta}{2} \int_0^t  \|X_s^{(n)}\|_H^{p-2}\|X_s^{(n)}\|_V^{\alpha}\d s \\
& + C \int_0^t \left( \|X_s^{(n)}\|_H^{p}
+ f_s\cdot \|X_s^{(n)}\|_H^{p-2}\right)\d s\\
&+ p\int_0^t\|X_s^{(n)}\|_H^{p-2}\<X_s^{(n)}, P_nB(s,X_s^{(n)})\d
W_s^{(n)}\>_H\\
\le & \|X_0\|_H^p - \frac{p\theta}{2}\int_0^t  \|X_s^{(n)}\|_H^{p-2}\|X_s^{(n)}\|_V^{\alpha}\d s\\
&+ C \int_0^t\left(\|X_s^{(n)}\|_H^{p}+ f_s^{p/2}\right)\d s\\
&+ p\int_0^t\|X_s^{(n)}\|_H^{p-2}\<X_s^{(n)}, P_nB(s,X_s^{(n)})\d
W_s^{(n)}\>_H, \  \ t\in[0,T],
\end{split}
\end{equation}
where $C$ is a generic constant (independent of $n$) and  may change
from line to line.

For any given $n$ we define the stopping time
$$  \tau_R^{(n)}=\inf\{t\in[0,T]: \|X_t^{(n)}\|_H>R   \}\wedge T , \ R>0. $$
Here we take $\inf\emptyset=\infty$. It's obvious that
$$ \lim_{R\rightarrow\infty} \tau_R^{(n)}=T,\  \P-a.s.,\  n\in\mathbb{N}.  $$

Then by  the Burkholder-Davis-Gundy inequality  we have
\begin{equation}\begin{split}\label{estimate 3}
  & \E\sup_{r\in[0,t]}\left|\int_0^r \|X_s^{(n)}\|_H^{p-2} \<X_s^{(n)},P_n B(s,X_s^{(n)})\d W_s^{(n)}\>_H  \right| \\
   \le & 3\E\left(\int_0^t\|X_s^{(n)}\|_H^{2p-2} \|B(s,X_s^{(n)})\|_{2}^2 \d s  \right)^{1/2}\\
   \le & 3\E\left( \sup_{s\in[0,t]}\|X_s^{(n))}\|_H^{2p-2}\cdot
C\int_0^t \left(\|X_s^{(n))}\|_H^2+f_s\right) \d s  \right)^{1/2}\\
   \le & 3 \E \left[ \varepsilon \sup_{s\in[0,t]}\|X_s^{(n)}\|_H^{p}+ C_\varepsilon \left( \int_0^t (\|X_s^{(n)}\|_H^2+f_s) \d s  \right)^{p/2}    \right] \\
\le & 3 \varepsilon \E\sup_{s\in[0,t]}\|X_s^{(n)}\|_H^{p}+3\cdot
(2T)^{p/2 -1} C_\varepsilon \E\int_0^t \left(\|X_s^{(n)}\|_H^p+
f_s^{p/2} \right)\d s,\ t\in[0, \tau_R^{(n)}],
\end{split}
\end{equation}
where $\varepsilon>0$ is a small constant and $C_\varepsilon$ comes from  Young's
inequality.

 Then by  $(\ref{Ito estimate 2})$,   $(\ref{estimate 3})$  and Gronwall's
lemma   we have
$$\E\sup_{t\in[0, \tau_R^{(n)}]}\|X_t^{(n)}\|_H^p +\E\int_0^{\tau_R^{(n)}}  \|X_s^{(n)}\|_H^{p-2}\|X_s^{(n)}\|_V^{\alpha}\d s
\le C\left(\E\|X_0\|_H^p+ \E\int_0^T f_s^{p/2}\d s\right), \ n\ge 1,$$
where $C$ is a constant independent of $n$.

For $R\rightarrow\infty$, (\ref{l2}) follows from the monotone
convergence theorem.

Moreover, by $(H4)$ and $p\ge \beta+2$ we have
$$  \|A(\cdot, X^{(n)})\|_{K^*} \le C, \ n\ge 1,     $$
where $C$ is a constant independent of $n$.
\end{proof}

\textbf{Proof of Theorem \ref{T1}.}
(1) Existence: By Lemmas \ref{L1} and \ref{L2} there exists a
subsequence $n_k\rightarrow \infty$ such that

(i) $X^{(n_k)}\rightarrow \bar{X}$ weakly in $K$ and weakly star in
$L^p(\Omega; L^\infty([0,T];H))$.

(ii) $Y^{(n_k)}:=A(\cdot,X^{(n_k)})\rightarrow Y$ weakly in $K^*$.

(iii) $Z^{(n_k)}:=P_{n_k} B(\cdot,X^{(n_k)})\rightarrow Z$ weakly in
$J$ and hence
$$    \int_0^\cdot P_{n_k} B(s,X^{(n_k)}_s)\d W_s^{(n_k)}\rightarrow \int_0^\cdot Z_s \d W_s     $$
weakly in $L^\infty([0,T], \d t; L^2(\Omega,\P; H)) $.

Now we define \begin{equation}
 X_t:=X_0+\int_0^t Y_s \d
s+\int_0^t Z_s \d W_s, \ t\in [0,T],
\end{equation}
it is easy to show that $X=\bar{X} \  \d t\otimes\P$-a.e.

Then by \cite[Theorem 4.2.5]{PR07} we know that $X$ is an $H$-valued continuous $(\mathcal{F}_t)$-adapted
process and
$$\E\left(\sup_{t\in[0,T]}\|X_t\|_H^p+\int_0^T \|X_t\|_V^\alpha \d t \right)   < \infty.$$
Therefore, it remains to verify that
$$   A(\cdot,\bar{X})=Y, \   B(\cdot,\bar{X})=Z \       \d t\otimes\P -a.e.  $$
Define
$$ \mathcal{M}=\bigg\{\phi: \phi\ \text{is}\ V\text{-valued}\ (\mathcal{F}_t)\text{-adapted process such that}\   \E\int_0^T\rho(\phi_s) ds<\infty  \bigg\}.   $$

For $  \phi\in K\cap\mathcal{M} \cap L^p(\Omega; L^\infty([0,T];H))$,
\begin{align}
& \E\left( e^{-\int_0^t(K+\rho(\phi_s)  )\d
s}\|X_t^{(n_k)}\|_H^2 \right)- \E\left(\|X_0^{(n_k)}\|_H^2\right)\\
\nonumber
 =& \E \bigg[ \int_0^t
e^{-\int_0^s(K+\rho(\phi_r)  )\d r} \bigg(
 2{ }_{V^*}\<A(s,X^{(n_k)}_s), X_s^{(n_k)}  \>_V +\|P_{n_k}B(s,X_s^{(n_k)})\tilde{P}_{n_k}\|_2^2\\
 \nonumber
& -(K+\rho(\phi_s) )\|X_s^{(n_k)}\|_H^2
 \bigg)\d s \bigg]\\ \nonumber
 \le & \E \bigg[ \int_0^t
e^{-\int_0^s(K+\rho(\phi_r)  )\d r} \bigg(
 2{ }_{V^*}\<A(s,X^{(n_k)}_s), X_s^{(n_k)}  \>_V +\|B(s,X_s^{(n_k)})\|_2^2\\
 \nonumber
& -(K+\rho(\phi_s) )\|X_s^{(n_k)}\|_H^2
 \bigg)\d s \bigg]\\ \nonumber
=& \E \bigg[ \int_0^t e^{-\int_0^s(K+\rho(\phi_r)  )\d r} \bigg(
 2{ }_{V^*}\<A(s,X^{(n_k)}_s)-A(s,\phi_s), X_s^{(n_k)}-\phi_s  \>_V \\
 \nonumber
 &+\|B(s,X_s^{(n_k)})-B(s,\phi_s)\|_2^2-(K+\rho(\phi_s) )\|X_s^{(n_k)}
 -\phi_s\|_H^2
 \bigg)\d s \bigg]\\ \nonumber
&+\E \bigg[ \int_0^t e^{-\int_0^s(K+\rho(\phi_r)  )\d r} \bigg(
 2{ }_{V^*}\<A(s,X^{(n_k)}_s)-A(s,\phi_s), \phi_s  \>_V + 2{ }_{V^*}\<A(s,\phi_s), X_s^{(n_k)}
 \>_V\\ \nonumber
& -\|B(s,\phi_s)\|_2^2+2\<B(s,X_s^{(n_k)}), B(s,\phi_s)
\>_{L_2(U,H)}\\ \nonumber &
 -2(K+\rho(\phi_s) )\<X_s^{(n_k)}, \phi_s\>_H
+(K+\rho(\phi_s) )\|\phi_s\|_H^2
 \bigg)\d s \bigg].
\end{align}
Let $k\rightarrow\infty$, by $(H2)$ and the lower semicontinuity
(cf. e.g. \cite[(4.2.27)]{PR07} for details) we have for every
nonnegative $ \psi\in L^\infty([0,T]; \d t)$,
\begin{align}\label{e9}
& \E\left[\int_0^T \psi_t \left( e^{-\int_0^t(K+\rho(\phi_s)
)\d s}\|X_t\|_H^2 - \|X_0\|_H^2  \right)\d t \right]\\
\nonumber
 \le & \liminf_{k\rightarrow\infty} \E\left[\int_0^T \psi_t \left( e^{-\int_0^t(K+\rho(\phi_s)
)\d s}\|X_t^{(n_k)}\|_H^2 - \|X_0^{(n_k)}\|_H^2  \right)\d t \right]\\
\nonumber \le &\E \bigg[\int_0^T \psi_t \bigg( \int_0^t
e^{-\int_0^s(K+\rho(\phi_r)  )\d r} \bigg(
 2{ }_{V^*}\<Y_s-A(s,\phi_s), \phi_s  \>_V \\ \nonumber
& + 2{ }_{V^*}\<A(s,\phi_s), \bar{X}_s
 \>_V -\|B(s,\phi_s)\|_2^2+2\<Z_s, B(s,\phi_s) \>_{L_2(U,H)}\\
\nonumber &
 -2(K+\rho(\phi_s) )\<X_s, \phi_s\>_H
+(K+\rho(\phi_s) )\|\phi_s\|_H^2
 \bigg)\d s \bigg) \d t \bigg].
\end{align}
By It\^o's formula we have for  $  \phi\in K\cap\mathcal{M} \cap L^p(\Omega; L^\infty([0,T];H))$,
\begin{align}\label{e3}
& \E\left( e^{-\int_0^t(K+\rho(\phi_s) )\d s}\|X_t\|_H^2
\right)- \E\left(\|X_0\|_H^2\right)\\ \nonumber
 =& \E \left[ \int_0^t
e^{-\int_0^s(K+\rho(\phi_r) )\d r} \left(
 2{ }_{V^*}\<Y_s, \bar{X}_s  \>_V +\|Z_s\|_2^2-(K+ \rho(\phi_s) )\|X_s\|_H^2 \right)\d s
\right].
\end{align}

By inserting (\ref{e3}) into (\ref{e9}) we obtain
\begin{align}
0\ge & \E\bigg[\int_0^T \psi_t \bigg( \int_0^t
e^{-\int_0^s(K+\rho(\phi_r)  )\d r}(
 2{ }_{V^*}\<Y_s-A(s,\phi_s), \bar{X}_s- \phi_s  \>_V \\ \nonumber
 & +\|B(s,\phi_s)-Z_s\|_2^2
 -(K+\rho(\phi_s) )\|X_s- \phi_s\|_H^2 )\d s \bigg) \d t \bigg].
\end{align}
Note that (\ref{c3}), Lemmas \ref{L1} and \ref{L2} imply that
$$ \bar{X}\in K\cap\mathcal{M} \cap L^p(\Omega; L^\infty([0,T];H)). $$
By taking $\phi=\bar{X}$
 we obtain that $Z=B(\cdot,\bar{X})$. Next,
 we first take $\phi=\bar{X}-\varepsilon\tilde{\phi} v$ for
$\tilde{\phi}\in L^\infty([0,T]\times\Omega; \d t\otimes\P;
\mathbb{R})$ and $v\in V$, then we divide by $\varepsilon$ and let
$\varepsilon\rightarrow 0$ to derive that
\begin{equation}
0\ge \E\bigg[\int_0^T \psi_t \bigg( \int_0^t
e^{-\int_0^s(K+\rho(\bar{X}_r) )\d r}\tilde{\phi}_s
 { }_{V^*}\<Y_s-A(s,\bar{X}_s), v  \>_V \d s \bigg) \d t \bigg].
 \end{equation}
By the arbitrariness of $\psi$ and $\tilde{\phi}$, we conclude that
$ Y=A(\cdot, \bar{X}) $.

Hence $\bar{X}$ is a solution of (\ref{SEE}).

(2) Uniqueness: Suppose $X_t,Y_t$ are the solutions of (\ref{SEE}) with initial conditions
$X_0,Y_0$ respectively, i.e.
 \begin{equation}\begin{split}
 X_t&=X_0+\int_0^t A(s,X_s) \d s+\int_0^t B(s,X_s) \d W_s, \ t\in[0,T];\\
  Y_t&=Y_0+\int_0^t A(s,Y_s) \d s+\int_0^t B(s,Y_s) \d W_s, \ t\in[0,T].
\end{split}
\end{equation}
Then by the product rule, It\^{o}'s formula and $(H2)$ we have
 \begin{equation*}\begin{split}
 & e^{-\int_0^t(K+\rho(Y_s))\d s}\|X_t-Y_t\|_H^2\\
\le & \|X_0-Y_0\|_H^2
+2\int_0^t  e^{-\int_0^s(K+\rho(Y_r))\d r} \<X_s-Y_s, B(s,X_s)\d W_s-B(s,Y_s)\d W_s  \>_H,\ t\in[0,T].
\end{split}
\end{equation*}
By a standard localization argument we have
$$ \E\left[ e^{-\int_0^t(K+\rho(Y_s))\d s}\|X_t-Y_t\|_H^2 \right] \le \E \|X_0-Y_0\|_H^2, \  t\in[0,T].   $$
If $X_0=Y_0, \P -a.s.$, then
$$ \E\left[e^{-\int_0^t(K+\rho(Y_s))\d s}\|X_t-Y_t\|_H^2\right]=0, \  t\in[0,T].  $$
Since (\ref{c3}) and Lemmas \ref{L1}, \ref{L2} imply that
$$ \int_0^t(K+\rho(Y_s))\d s< \infty, \ \P-a.s., \ t\in[0,T],  $$
we have
$$  X_t=Y_t, \ \P-a.s., \ t\in[0,T]. $$
Therefore, the pathwise uniqueness follows from the path continuity of $X,Y$ in $H$.

(3) Markov property: the proof of Markov property is standard, we refer to \cite[Proposition 4.3.5]{PR07}
or \cite[Theorem II.2.4]{KR79}.
\qed

 \section{Application to examples}
Obviously, the main result can be applied to  stochastic evolution
equations with monotone coefficients (cf. \cite{PR07} for the
stochastic porous medium equation and stochastic $p$-Laplace
equation) and non-monotone perturbations ($e.g.$ some locally
Lipschitz perturbation) in the drift. Below we  present some
examples where the coefficients are only locally monotone, hence the
classical result of monotone operators cannot be applied.

 In this section  we use the notation $D_i$ to denote the spatial derivative $\frac{\partial}{\partial x_i}$,
$\Lambda \subseteq \mathbb{R}^d$ is an open bounded domain with smooth
boundary.
 For the standard Sobolev space $W_0^{1,p}(\Lambda)$ $(p\ge 2)$  we always use the following (equivalent) Sobolev norm:
$$    \|u\|_{1,p}:=\left(\int_\Lambda |\nabla u(x)|^p d x\right)^{1/p}.      $$
For simplicity we only consider  examples where the coefficients are
time independent, but one can easily adapt those examples to the
time dependent case.


\begin{lem}\label{L3.1}
Consider the Gelfand triple
$$ V:=W_0^{1,2}(\Lambda)\subseteq H:=L^2(\Lambda) \subseteq  W^{-1,2}(\Lambda)  $$
and the operator
$$ A(u)=\Delta u+ \sum_{i=1}^d f_i(u)D_i u,  $$
where $f_i$ ($i=1,\cdots,d$) are bounded Lipschitz functions on $\mathbb{R}$.

$(1)$ If $d<3$,  then there exists a constant
$K$ such that
$$2 { }_{V^*}\<A(u)-A(v), u-v\>_V
   \le - \|u-v\|_V^2+  \left(K  + K\|v\|_V^2  \right)\|u-v\|_H^2,\ u,v\in V.$$

$(2)$ If $d=3$, then there exists a constant
$K$ such that
$$2 { }_{V^*}\<A(u)-A(v), u-v\>_V
   \le - \|u-v\|_V^2+  \left(K + K\|v\|_V^4 \right)\|u-v\|_H^2,\ u,v\in V.$$

$(3)$ If $f_i$ are  independent of $u$ for $i=1,\cdots,d$, i.e.
$$ A(u)=\Delta u+ \sum_{i=1}^d f_i\cdot D_i u,  $$
then for any $d\ge 1$ we have
$$2 { }_{V^*}\<A(u)-A(v), u-v\>_V
   \le - \|u-v\|_V^2+  K\|u-v\|_H^2,\ u,v\in V.$$
\end{lem}
\begin{proof}  (1)
 Since all $f_i$ are bounded and Lipschitz, we have
\begin{equation}
 \begin{split}
\label{e3.1}
& ~~~~ { }_{V^*}\<A(u)-A(v), u-v\>_V \\
 &= - \|u-v\|_V^2+ \sum_{i=1}^d \int_\Lambda \left(f_i(u)D_i u-f_i(v)D_i v\right)\left(u-v\right) d
 x\\
&= - \|u-v\|_V^2+ \sum_{i=1}^d \int_\Lambda \left(f_i(u)(D_i u-D_i
v) +D_iv( f_i(u) -f_i(v))\right)\left(u-v\right) dx\\
&\le - \|u-v\|_V^2+ \sum_{i=1}^d \bigg[ \left( \int_\Lambda (D_i
u-D_i v)^2 d x\right)^{1/2} \left( \int_\Lambda f_i^2(u)(u-v)^2 dx
\right)^{1/2}\\
&~~ +\left( \int_\Lambda (D_i v)^2 d x\right)^{1/2} \left(
\int_\Lambda \left(f_i(u)-f_i(v)\right)^2
 (u-v)^2 dx \right)^{1/2} \bigg] \\
 &\le - \|u-v\|_V^2+ K \|u-v\|_V   \left(\int_\Lambda
 (u-v)^2 dx \right)^{1/2}    + K \| v\|_V \left(\int_\Lambda
 (u-v)^4 dx \right)^{1/2} \\
&\le -\frac{3}{4} \|u-v\|_V^2+ K \|u-v\|_H^2  + K \|v\|_V \|u-v\|_{L^4}^2,
 \ u,v\in V,
\end{split}
\end{equation}
where $K$ is a generic constant that may change from line to line.

For $d<3$, we have
the following well-known estimate on $\mathbb{R}^2$ (see \cite[Lemma
2.1]{MS02})
\begin{equation}\label{2d}
  \|u\|_{L^4}^4 \le 2   \|u\|_{L^2}^2  \|\nabla u\|_{L^2}^2, \ u\in W_0^{1,2}(\Lambda) .
\end{equation}
Hence combining with (\ref{e3.1}) we have
$$ { }_{V^*}\<A(u)-A(v), u-v\>_V \le -\frac{1}{2} \|u-v\|_V^2+  \left(K  + K\|v\|_V^2  \right)\|u-v\|_H^2,\ u,v\in V.$$

$(2)$ For $d=3$ we use the following estimate (cf.\cite{MS02})
\begin{equation}\label{e4}
  \|u\|_{L^4}^4 \le 4  \|u\|_{L^2}  \|\nabla u\|_{L^2}^3, \  u\in W_0^{1,2}(\Lambda),
\end{equation}
then the second assertion can be derived similarly from (\ref{e3.1}) and Young's inequality.

$(3)$ This assertion can be easily derived as (\ref{e3.1}).
\end{proof}

\beg{exa} (Semilinear stochastic equations)\\
Let $\Lambda$ be an open bounded domain in $\mathbb{R}^d$ with smooth boundary. We consider
the following triple
$$ V:=W^{1,2}_0(\Lambda) \subseteq H:=L^2(\Lambda)\subseteq (W^{1,2}_0(\Lambda))^*$$
and the semilinear stochastic  equation
 \begin{equation}\label{rde}
 \d X_t=\left(\Delta X_t+ \sum_{i=1}^d f_i(X_t)D_i X_t+ g(X_t)\right)\d t+B(X_t)\d W_t,
  \end{equation}
where  $W_t$ is a Wiener process on $L^2(\Lambda)$ and $f_i,g,B$ satisfy
 the following conditions:

(i)  $f_i$ are bounded  Lipschitz functions on $\mathbb{R}$ for $i=1,\cdots, d$;

(ii) $g$ is a continuous function  on $\mathbb{R}$ such that
\begin{equation}\label{c1}
\begin{split}
|g(x)| & \le C(|x|^r+1), \  x\in \mathbb{R};\\
 (g(x)-g(y))(x-y)&\le C(1+|y|^s)(x-y)^2, \  x,y\in \mathbb{R}.
\end{split}
     \end{equation}
where $C,r,s$ are some positive constants.

(iii) $B:V\rightarrow L_2(L^2(\Lambda))$ is Lipschitz.

Then we have the following result:

(1) If $d=1,r=3,s=2$, then for
 any $X_0\in L^{6}(\Omega, \mathcal{F}_0,\mathbb{P};H)$,
    $(\ref{rde})$
    has a unique solution $\{X_t\}_{t\in [0,T]}$ and this solution satisfies
\begin{equation}\label{solution estimate}
\E\left(\sup_{t\in[0,T]}\|X_t\|_H^6+\int_0^T \|X_t\|_V^2 \d t \right)   < \infty.
\end{equation}

(2) If $d=2,r=\frac{7}{3},s=2$, then for
 any $X_0\in L^{6}(\Omega, \mathcal{F}_0,\mathbb{P};H)$,
    $(\ref{rde})$
    has a unique solution $\{X_t\}_{t\in [0,T]}$ and this solution satisfies (\ref{solution estimate}).

(3) If $d= 3,r=\frac{7}{3},s=\frac{4}{3}$, $f_i, i=1,\cdots,d$ are
bounded measurable functions and independent of $u$,
 then for
 any $X_0\in L^{6}(\Omega, \mathcal{F}_0,\mathbb{P};H)$,
    $(\ref{rde})$
    has a unique solution $\{X_t\}_{t\in [0,T]}$ and this solution satisfies (\ref{solution estimate}).
\end{exa}

\begin{proof}
(1)  We define the operator
$$ A(u)=\Delta u+ \sum_{i=1}^d f_i(u)D_i u+g(u), \ u\in V.  $$

The hemicontinuity $(H1)$ follows easily from the continuity of $f$ and $g$.

Note that (\ref{c1}) and (\ref{2d}) imply
\begin{equation}
 \begin{split}
  { }_{V^*}\<g(u)-g(v),u-v\>_V
\le & C\left(1+ \|v\|_{L^{2s}}^s \right)\|u-v\|_{L^4}^2\\
\le & \frac{1}{4}\|u-v\|_V^2+C\left(
1+\|v\|_{L^{2s}}^{2s} \right) \|u-v\|_H^2, \
u,v\in V.
 \end{split}
\end{equation}
Therefore, by Lemma \ref{L3.1} we have for $d<3$
$$2 { }_{V^*}\<A(u)-A(v), u-v\>_V
   \le -\frac{1}{2} \|u-v\|_V^2+ C \left(1  + \|v\|_{V}^2 +\|v\|_{L^{2s}}^{2s}  \right)\|u-v\|_H^2,\ u,v\in V,$$
$i.e.$ $(H2)$, $(H3)$ hold with $\rho(v)= \|v\|_{V}^2
+\|v\|_{L^{2s}}^{2s} $ and  $\alpha=2$.

For $d=1, r=3$, by  the Sobolev embedding theorem we have
$$  \|g(u)\|_{V^*}\le C\left(1+\|u\|_{L^3}^3\right)\le C\left(1+\|u\|_V\|u\|_{H}^{2}\right), \ u\in V. $$
Then it is easy to show that
$$  \|A(u)\|_{V^*}\le  C\left(1+\|u\|_V+  \|u\|_V\|u\|_{H}^{2}\right), \ u\in V. $$
 Hence $(H4)$  holds with $\beta=4$.

Therefore, all assertions follow from Theorem \ref{T1} by taking
$p=6$.

(2) For $d=2,3$ we have
 $$  \|g(u)\|_{V^*}\le C\left(1+\|u\|_{L^{6r/5}}^{r}\right), \ u\in V. $$

For  $r= \frac{7}{3}$, by the interpolation theorem we have
$$ \|u\|_{L^{6r/5}}\le \|u\|_{L^2}^{4/7} \|u\|_{L^6}^{3/7}, \ u\in W_0^{1,2}(\Lambda) \subseteq L^6(\Lambda).  $$
Then
\begin{equation}\label{e5}  \|g(u)\|_{V^*}\le
C\left(1+\|u\|_{L^{6r/5}}^{r}\right)\le C\left(1+ \|u\|_H^{4/3}
\|u\|_{V} \right), \ u\in V.
\end{equation}
Hence $(H4)$ holds for $d=2,3$ with $\beta=8/3$.

Therefore, for $d=2$, all assertions follow from Theorem \ref{T1} by
taking $p=6$ (in fact, $p=14/3$ is enough).

(3) If $d=3$ and $f_i, i=1,2,3$ are bounded measurable functions and
independent of $u$, then by Lemma \ref{L3.1} and (\ref{e4})  we have
$$2 { }_{V^*}\<A(u)-A(v), u-v\>_V
   \le - \frac{1}{2}\|u-v\|_V^2+
K\left(1  +  \|v\|_{L^{2s}}^{4s} \right)\|u-v\|_H^2,\ u,v\in V.$$
Hence $(H2)$, $(H3)$ hold with $\rho(v)= \|v\|_{L^{2s}}^{4s}$ and
$\alpha=2$.

Since $s=\frac{4}{3}$, by the
interpolation inequality we have
$$   \|u\|_{L^{2s}} \le   \|u\|_{L^2}^{5/8}   \|u\|_{L^{6}}^{3/8}, \ u\in V.  $$
Therefore,
$$   \|u\|_{L^{2s}}^{4s} \le  C \|u\|_{H}^{10/3}   \|u\|_{V}^{2}, \ u\in V,  $$
i.e. (\ref{c3}) holds with $\beta=10/3$.

Hence combining with (\ref{e5}) we can  take $p=6$ (in fact, $p\ge
16/3$ is enough).

Then all assertions follow from Theorem \ref{T1}.
\end{proof}

\begin{rem}
(1) For some specific examples, one might derive the local
monotonicity without assuming the boundedness of $f_i, i=1,\cdots,
d$.  For instance,  Wilhelm Stannat (whom we like to thank for this
at this point) pointed out to us that our local monotonicity
condition is also fulfilled by the classical stochastic Burgers
equation. Since the remaining conditions hold anyway in this case,
all our results apply to the classical stochastic Burgers equation
as well. More precisely, for the classical stochastic Burgers
equation we have
$$d=1,\ \Lambda=[0, 1],\  A(u)=\Delta u+ u
\frac{\partial u}{\partial x},$$
 then we can derive the following local
monotonicity:
\begin{equation}
 \begin{split}
& ~~~~ { }_{V^*}\<A(u)-A(v), u-v\>_V \\
 &= - \|u-v\|_V^2+  \int_\Lambda \left( u \frac{\partial u}{\partial x} - v \frac{\partial v}{\partial x}\right)\left(u-v\right) d
 x\\
&= - \|u-v\|_V^2- \frac{1}{2} \int_\Lambda \left( u-v +2v \right)\left(u-v \right) \frac{\partial }{\partial x} \left(u-v\right) dx\\
&= - \|u-v\|_V^2- \int_\Lambda  v \left( u-v \right)\frac{\partial }{\partial x} \left(u-v\right) dx\\
&\le - \|u-v\|_V^2+  \|v\|_{L^4} \|u-v\|_{L^4} \|u-v\|_{V}\\
 &\le - \|u-v\|_V^2+ K \|v\|_{L^4} \|u-v\|_{H}^{1/2} \|u-v\|_{V}^{3/2} \\
&\le -\frac{3}{4} \|u-v\|_V^2+ K \|v\|_{L^4}^4 \|u-v\|_{H}^{2},
 \ u,v\in V,
\end{split}
\end{equation}
where $K$ is  some constant that may change from line to line.

(2) One obvious generalization is that one can replace  $\Delta $ in
(\ref{rde}) by the $p$-Laplace operator $ \mathbf{div}(|\nabla
u|^{p-2}\nabla u)$ or the more general quasi-linear differential
operator
$$ \sum_{|\alpha|\le m} (-1)^{|\alpha|}D_\alpha A_\alpha(D u),  $$
where $Du=(D_\beta u)_{|\beta|\le m}$. Under certain assumptions
(cf. \cite[Proposition 30.10]{Z90}) this operator satisfies the
monotonicity and coercivity condition. Then, according to Theorem
\ref{T1}, we can  obtain the existence and uniqueness of solutions
to this type of quasi-linear  SPDE with  non-monotone perturbations
($e.g.$ some locally Lipschitz lower order terms).
\end{rem}

Now we apply Theorem \ref{T1} to the stochastic 2-D Navier-Stokes equation.

Let $\Lambda$ be a bounded domain in $\mathbb{R}^2$ with smooth boundary. Define
$$ V=\left\{ v\in W_0^{1,2}(\Lambda,\mathbb{R}^2): \nabla \cdot v=0 \  a.e.\  \text{in} \ \Lambda   \right\}, \
\|v\|_V:=\left(\int_\Lambda |\nabla v|^2 dx  \right)^{1/2},
$$
and $H$ is the closure of $V$ in the following norm
$$ \|v\|_H:=\left(\int_\Lambda | v|^2 dx  \right)^{1/2}.$$
The linear operator $P_H$ (Helmhotz-Hodge projection) and $A$ (Stokes operator with viscosity constant
$\nu$) are defined by
$$ P_H: L^2(\Lambda, \mathbb{R}^2)\rightarrow H\ \  \text{ orthogonal projection}; $$
$$  A: W^{2,2}(\Lambda, \mathbb{R}^2)\cap V\rightarrow H, \ Au=\nu P_H \Delta u .  $$
It is well known then  the Navier-Stokes equation can be
reformulated as follows
\begin{equation}\label{NSE}
u'=Au+F(u)+f,\ u(0)=u_0\in H,
\end{equation}
where $f\in L^2(0,T;V^*)$ denotes some external force and
$$ F:  \mathcal{D}_F\subset H\times V\rightarrow H, \ F(u,v)=- P_H\left[\left(u \cdot \nabla\right) v\right],
F(u)=F(u,u).  $$ It is standard that  using the Gelfand triple
$$     V\subseteq H\equiv H^*\subseteq V^*,   $$
we see that the following mappings
$$ A: V\rightarrow V^*, \  F: V\times V\rightarrow V^*  $$
are well defined. In particular, we have
$$ { }_{V^*}\<F(u,v),w\>_V=-{ }_{V^*}\<F(u,w),v\>_V, \   { }_{V^*}\<F(u,v),v\>_V=0,\  u,v,w\in V. $$
Now we consider the stochastic 2-D Navier-Stokes equation
\begin{equation}\label{SNSE}
\d X_t=\left(AX_t+F(X_t)+f_t\right)\d t+ B(X_t)\d W_t,
\end{equation}
where  $W_t$ is a Wiener process on $H$.

\begin{exa}(Stochastic 2-D Navier-Stokes equation)
Suppose that $X_0\in L^{4}(\Omega, \mathcal{F}_0,\mathbb{P}; H)$ and
$B:V\rightarrow L_2(H)$ satisfies
$$ \|B(v_1)-B(v_2)\|_2^2\le K\left(1+\|v_2\|_{L^4\left(\Lambda;\mathbb{R}^2\right)}^4 \right)
\|v_1-v_2\|_H^2, \ v_1,v_2\in V, $$ where $K$ is some constant.
 Then
    $(\ref{SNSE})$
    has a unique solution $\{X_t\}_{t\in [0,T]}$ and this solution satisfies
\begin{equation*}
\E\left(\sup_{t\in[0,T]}\|X_t\|_H^4+\int_0^T \|X_t\|_V^2 \d t \right)   < \infty.
\end{equation*}
\end{exa}

\begin{proof} The hemicontinuity $(H1)$ is obvious since $F$ is a bilinear map.

Note that $ { }_{V^*}\<F(v),v\>_V=0$, it is also easy to show $(H3)$
with $\alpha=2$:
\begin{equation*}\begin{split}
 { }_{V^*}\<Av+F(v)+f_t,v\>_V &\le -\nu\|v\|_V^2+\|f_t\|_{V^*}\|v\|_V \le
  -\frac{\nu}{2}\|v\|_V^2+C\|f_t\|_{V^*}^2, \ v\in V,\\
  \|B(v)\|_2^2 &\le 2K\|v\|_H^2+2\|B(0)\|_2^2, \ v\in V.
\end{split}
\end{equation*}
 Recall the following estimates (cf.
e.g.\cite[Lemmas 2.1, 2.2]{MS02})
\begin{equation}\label{e2}
 \begin{split}
  |{ }_{V^*}\<F(w),v\>_V| &\le 2 \|w\|_{L^4(\Lambda;\mathbb{R}^2)}\|v\|_V; \\
  |{ }_{V^*}\<F(w),v\>_V| &\le 2  \|w\|_V^{3/2} \|w\|_H^{1/2}  \|v\|_{L^4(\Lambda;\mathbb{R}^2)}, v,w\in V. \\
 \end{split}
\end{equation}
Then we have
\begin{equation}
 \begin{split}
  { }_{V^*}\<F(u)-F(v),u-v\>_V &=-  { }_{V^*}\<F(u,u-v),v\>_V+  { }_{V^*}\<F(v,u-v),v\>_V \\
 &= -  { }_{V^*}\<F(u-v),v\>_V \\
 &\le 2  \|u-v\|_V^{3/2} \|u-v\|_H^{1/2}  \|v\|_{L^4(\Lambda;\mathbb{R}^2)} \\
& \le \frac{\nu}{2} \|u-v\|_V^{2} + \frac{32}{\nu^3}  \|v\|_{L^4(\Lambda;\mathbb{R}^2)}^4 \|u-v\|_H^{2},
\ u,v\in V.
 \end{split}
\end{equation}
Hence we have the local monotonicity $(H2)$ with $\rho(v)=
\|v\|_{L^4(\Lambda;\mathbb{R}^2)}^4 $:
$$  { }_{V^*}\<Au+F(u)-Av-F(v),u-v\>_V
\le -\frac{\nu}{2} \|u-v\|_V^{2} + \frac{32}{\nu^3}  \|v\|_{L^4(\Lambda;\mathbb{R}^2)}^4 \|u-v\|_H^{2}.  $$

(\ref{e2}) and (\ref{2d}) imply that  $(H4)$ and (\ref{c3}) hold
with $\beta=2$.

Therefore, the existence and uniqueness of solutions to (\ref{SNSE})
follow from Theorem \ref{T1} by taking $p=4$.
\end{proof}

\begin{rem} (1) If the noise in (\ref{SNSE}) is additive type, then the
existence and uniqueness of solutions to (\ref{SNSE}) have been
established in \cite{MS02}. Here we can conclude the same result for
(\ref{SNSE}) with general multiplicative noise  by a direct
application of our main result.

(2) For the 3-D Navier-Stokes equation, we recall the following well-known estimate (cf.  e.g. \cite[(2.5)]{MS02})
$$   \|\psi\|_{L^4}^4\le 4 \|\psi\|_{L^2} \|\nabla \psi\|_{L^2}^3, \ \psi\in W_0^{1,2}(\Lambda; \mathbb{R}^3).   $$
Then  one can show  that
\begin{equation*}
 \begin{split}
  { }_{V^*}\<F(u)-F(v),u-v\>_V &= -  { }_{V^*}\<F(u-v),v\>_V \\
 &\le 2  \|u-v\|_V^{7/4} \|u-v\|_H^{1/4}  \|v\|_{L^4(\Lambda;\mathbb{R}^3)} \\
& \le \frac{\nu}{2} \|u-v\|_V^{2} + \frac{2^{12}}{\nu^7}  \|v\|_{L^4(\Lambda;\mathbb{R}^3)}^8 \|u-v\|_H^{2},
\ u,v\in V.
 \end{split}
\end{equation*}
Hence we have the following local monotonicity $(H2)$:
$$  { }_{V^*}\<Au+F(u)-Av-F(v),u-v\>_V
 \le -\frac{\nu}{2} \|u-v\|_V^{2} + \frac{2^{12}}{\nu^7}  \|v\|_{L^4(\Lambda;\mathbb{R}^3)}^8 \|u-v\|_H^{2}.  $$

 Another form of local monotonicity  can be derived similarly:
\begin{equation*}
 \begin{split}
  { }_{V^*}\<F(u)-F(v),u-v\>_V &= -  { }_{V^*}\<F(u-v),v\>_V \\
 &\le 2  \|u-v\|_V^{3/2} \|u-v\|_H^{1/2}  \|v\|_{L^6(\Lambda;\mathbb{R}^3)} \\
& \le \frac{\nu}{2} \|u-v\|_V^{2} + \frac{32}{\nu^3}  \|v\|_{L^6(\Lambda;\mathbb{R}^3)}^4 \|u-v\|_H^{2},
\ u,v\in V.
 \end{split}
\end{equation*}

(3) Concerning the growth condition, we have  in the 3-D case that
$$ \|F(u)\|_{V^*}\le 2 \|u\|_{L^4(\Lambda; \mathbb{R}^3)}^2\le 4\|u\|_H^{1/2}\|u\|_V^{3/2}, \ u\in V.$$

Unfortunately, this is not enough to verify  $(H4)$ in Theorem
\ref{T1}.

(4) One should note that the only role of $(H4)$ is to assure that
$\|A(\cdot, X^{(n)})\|_{K^*}$ is uniformly bounded for all $n$ (see
Lemma  \ref{L2}). Therefore, one can replace $(H4)$ by some weaker
growth condition   once we can derive some stronger a priori
estimate for  $X^{(n)}$ in the Galerkin approximation (e.g. as in
\cite{RZ10}). One good example of a further generalization of our
main result is that  we can apply Theorem \ref{T1} (with a revised
version of $(H4)$) to derive the existence and uniqueness of
solutions to the following stochastic tamed 3-D Navier-Stokes
equation with smooth enough initial condition:
\begin{equation*}
\d X_t=\left(AX_t+F(X_t)+f_t-P_H\left(g_N(|X_t|^2)
X_t\right)\right)\d t+ B(X_t)\d W_t,
\end{equation*}
where the taming function $g_N : \mathbb{R}_+\rightarrow
\mathbb{R}_+$ is smooth and satisfies for some $N > 0$
$$\begin{cases} & g_N(r)=0,\  \text{if}\  r\le N,\\
 & g_N(r)=(r-N)/\nu,\   \text{if}\  r\ge N+1,\\
 & 0\le g_N^\prime(r)\le C,\  \  r\ge 0.
\end{cases}
$$
 We refer to
\cite{RZ09,RZ10} for more details on the stochastic tamed 3-D
Navier-Stokes equation.
\end{rem}

\section{Appendix: The classical monotone and coercivity conditions}

For the existence and
uniqueness of the solution to (\ref{SEE}) we  recall the
following classical monotone and coercivity conditions on $A$ and $B$.

Suppose there exist constants $\alpha>1$,  $\theta>0$,
$K$ and a positive adapted process $f\in L^1([0,T]\times \Omega; \d
    t\times \mathbb{P})$ such that the
 following
 conditions hold for all $v,v_1,v_2\in V$ and $(t,\omega)\in [0,T]\times \Omega$.
\begin{enumerate}
    \item [$(A1)$] (Hemicontinuity) The map
     $ s\mapsto { }_{V^*}{ }_{V^*}\<A(t,v_1+s
 v_2),v\>_V$
  is  continuous on $\mathbb{R}$.

    \item [$(A2)$] (Monotonicity)
$$2 { }_{V^*}\<A(t,v_1)-A(t,v_2), v_1-v_2\>_V
    +\|B(t,v_1)-B(t,v_2)\|_{2}^2\\
     \le K\|v_1-v_2\|_H^2. $$

\item [$(A3)$] (Coercivity)
    $$ 2 { }_{V^*}\<A(t,v), v\>_V +\|B(t,v)\|_{2}^2 +\theta
    \|v\|_V^{\alpha} \le f_t +K\|v\|_H^2.$$

\item[$(A4)$] (Growth)
$$ \|A(t,v)\|_{V^*} \le f_t^{(\alpha-1)/\alpha} +
 K\|v\|_V^{\alpha-1}.$$
    \end{enumerate}

\beg{thm} {\bf (\cite{KR79} Theorems II.2.1, II.2.2)}\label{L0}
 Suppose $(A1)-(A4)$ hold,
    then for any $X_0\in L^2(\Omega, \mathcal{F}_0,\mathbb{P}; H)$
    $(\ref{SEE})$
    has a unique solution $\{X_t\}_{t\in [0,T]}$ and this solution satisfies
$$\E\sup_{t\in[0,T]}\|X_t\|_H^2< \infty.$$
Moreover, we have the following It\^{o}  formula
\begin{equation*}
\begin{split}
\|X_t\|_H^2=&\|X_0\|_H^2+\int_0^t\left( 2{ }_{V^*}\<A(s,X_s), X_s\>_V+\|B(s,X_s)\|_2^2\right)\d s\\
+&2\int_0^t\<X_s,B(s,X_s)\d W_s\>_H, \ t\in[0,T],\  \P-a.s.
\end{split}
\end{equation*}
\end{thm}

 \paragraph{Acknowledgement.}
This work is supported in part by
the SFB-701, the BiBoS-Research Center and NNSFC(10721091). The support of  Issac Newton Institute for Mathematical 
Sciences in Cambridge is also gratefully acknowledged where part of this work was done during the special semester on 
``Stochastic Partial Differential Equations''.



\begin{thebibliography}{10}


\bibitem{BDR07}
V.~Barbu, G.~Da~Prato and M.~R\"{o}ckner, \emph{{E}xistence of
strong
  solutions for stochastic porous media equation under general monotonicity
  conditions},  Ann. Probab. \textbf{37} (2009), 428--452.

\bibitem{BDR08}
V.~Barbu, G.~Da~Prato and M.~R\"{o}ckner, \emph{{E}xistence and
uniqueness of nonnegative solutions to the
  stochastic porous media equation}, Indiana Univ. Math. J. \textbf{57} (2008),
  no.~1, 187--212.

\bibitem{BDR2}
V.~Barbu, G.~Da~Prato and M.~R\"{o}ckner, \emph{{S}ome results on
stochastic porous media equations}, Bollettino
  U.M.I. \textbf{9} (2008), 1--15.

\bibitem{BDR1}
V.~Barbu, G.~Da~Prato and M.~R\"{o}ckner, \emph{{S}tochastic porous
media equation and self-organized
  criticality}, Commun. Math. Phys. \textbf{285} (2009), 901--923.

\bibitem{BGLR}
W.~Beyn, B.~Gess, P.~Lescot and M.~R\"{o}ckner, \emph{{T}he global
random attractor for a class of stochastic porous media equations},
Preprint.




\bibitem{Bro63}
F. E. Browder, \emph{Nonlinear elliptic boundary value problems},
Bull. Amer. Math. Soc.
  \textbf{69} (1963), 862--874.

\bibitem{Bro64}
F. E. Browder, \emph{Non-linear equations of Evolution}, Ann. Math.
\textbf{80} (1964), 485--523.




\bibitem{C92}
P.L. Chow, \emph{{L}arge deviation problem for some parabolic
{I}t\^{o}
  equations}, Comm. Pure Appl. Math. \textbf{45} (1992), 97--120.




\bibitem{DRRW}
G.~Da~Prato, M.~R\"{o}ckner, B.L. Rozovskii and F.-Y. Wang,
\emph{{S}trong
  solutions to stochastic generalized porous media equations: existence,
  uniqueness and ergodicity}, Comm. Part. Diff. Equa. \textbf{31} (2006),
  277--291.

\bibitem{DPZ92} G. Da Prato and J. Zabczyk, \emph{ Stochastic Equations
 in
Infinite Dimensions,} Encyclopedia of Mathematics and its
Applications, Cambridge University Press. 1992.



\bibitem{GLR}
 B.~Gess, W.~Liu and M.~R\"{o}ckner,
 \emph{Random attractors for a class of stochastic partial differential equations
 driven by general additive noise}, Preprint.

\bibitem{G}
I.~Gy{}{\"o}ngy, \emph{{O}n stochastic equations with respect to
semimartingale
  {III}}, Stochastics \textbf{7} (1982), 231--254.

\bibitem{GM07}
I.~Gy{}{\"o}ngy and A.~Millet, \emph{Rate of convergence of implicit
  approximations for stochastic evolution equations}, Interdiscip. Math. Sci.,
  vol.~2, pp.~281--310, World Sci. Publ., Hackensack, 2007.

\bibitem{GM09}
I.~Gy{}{\"o}ngy and A.~Millet, \emph{Rate of convergence of space
time approximations for stochastic evolution equations}, Pot. Anal.
\textbf{30} (2009),  29--64.




\bibitem{Li72}
J. L. Lions, \emph{{S}ome methods of solving nonlinear boundary
value
  problems}, Mir, Moscow, 1972, Russian translation.
  163--194.

\bibitem{KR79}
N.V. Krylov and B.L. Rozovskii, \emph{{S}tochastic evolution
equations},
  Translated from Itogi Naukii Tekhniki, Seriya Sovremennye Problemy Matematiki
  \textbf{14} (1979), 71--146.



\bibitem{LW08}
W.~Liu and F.-Y. Wang, \emph{{H}arnack inequality and strong
{F}eller property
  for stochastic fast diffusion equations}, J. Math. Anal. Appl. \textbf{342}
  (2008), 651--662.




\bibitem{L08}
W.~Liu, \emph{{H}arnack inequality and applications for stochastic
evolution
  equations with monotone drifts}, J. Evol. Equat. \textbf{9}(2009), 747--770.

\bibitem{L08b}
W.~Liu, \emph{{L}arge deviations for stochastic evolution equations
with small
  multiplicative noise}, Appl. Math. Optim. \textbf{61} (2010), 27--56.

\bibitem{L10} W. Liu,
    \emph{Invariance of subspaces under the solution flow of SPDE,}
   Infin. Dimens. Anal. Quantum Probab. Relat. Top. \textbf{13} (2010), 87--98.



\bibitem{L} W. Liu,
    \emph{Existence and Uniqueness of Solutions to Nonlinear Evolution Equations
 with Locally Monotone Operators,}
   Preprint.

\bibitem{MS02}
J.-L. Menaldi and S.S. Sritharan,  \emph{Stochastic 2-D
Navier-Stokes equation}, Appl. Math. Optim. \text{46}(2002), 31--53.


\bibitem{Mi62}
G.J. Minty, \emph{Monotone (non-linear) operators in Hilbert space},
Duke. Math. J.  \textbf{29} (1962), 341--346.

\bibitem{Mi63}
G.J. Minty, \emph{On a monotonicity method for the solution of
non-linear equations in Banach space},  Proc. Nat. Acad. Sci. USA
\textbf{50} (1963), 1038--1041.

\bibitem{Par75}
E.~Pardoux, \emph{{E}quations aux d\'eriv\'ees partielles
stochastiques non
  lin\'eaires monotones}, Ph.D. thesis, Universit\'e Paris XI, 1975.

\bibitem{PR07}
C.~Pr\'{e}v\^{o}t and M.~R\"{o}ckner, \emph{A concise course on
stochastic
  partial differential equations}, Lecture Notes in Mathematics, vol. 1905,
  Springer, 2007.

\bibitem{RRW}
J.~Ren, M.~R\"ockner and F.-Y. Wang, \emph{{S}tochastic generalized
porous
  media and fast diffusion equations}, J. Diff. Equa. \textbf{238} (2007),
  118--152.

\bibitem{RZ}
J.~Ren  and X. Zhang, \emph{{F}redlin-Wentzell's large deviation
principle for stochastic evolution equations}, J. Funct. Anal.
\textbf{254} (2008),
  3148--3172.

  \bibitem{RZ09}
M.~R\"ockner and X. Zhang, \emph{Stochastic tamed 3D Navier-Stokes
equations: existence, uniqueness and  ergodicity}, Probab. Theory
Related Fields  \textbf{145} (2009), no. 1-2, 211--267.




 \bibitem{RZ10}
M.~R\"ockner and T. Zhang, \emph{Stochastic tamed 3D Navier-Stokes
equations: existence, uniqueness and  small time large deviation
principles}, BiBoS-Preprint 10-03-339.

\bibitem{RW}
M.~R\"{o}ckner and F.-Y. Wang, \emph{{N}on-monotone stochastic
porous media
  equation}, J. Diff. Equa. \textbf{245} (2008), 3898--3935.


\bibitem{RWW}
M.~R\"ockner, F.-Y. Wang and L.~Wu, \emph{{L}arge deviations for
stochastic
  generalized porous media equations}, Stoc. Proc. Appl. \textbf{116} (2006),
  1677--1689.



\bibitem{Sh97}
R.E. Showalter, \emph{Monotone operators in Banach space and
nonlinear partial differential equations}, Mathematical Surveys and
Monographs 49, American Mathematical Society, Providence, 1997.



\bibitem{W07}
F.-Y. Wang,  \emph{{H}arnack inequality and applications for
stochastic generalized
  porous media equations}, Ann. Probab. \textbf{35} (2007), 1333--1350.

\bibitem{Zh08}
X.~Zhang, \emph{{O}n stochastic evolution equations with
non-{L}ipschitz coefficients}, Stoc. Dyna. \textbf{9} (2009),
549--595.


\bibitem{Z90}
E. Zeidler, \emph{Nonlinear Functional Analysis and its Applications
II/B: Nonlinear Monotone operators}, Springer-Verlag, New York,
1990.





\end{thebibliography}
\end{document}